\documentclass[letter]{amsart}
\usepackage{amsmath,amssymb,amsfonts,amsthm,calc,color,mathrsfs,bbm}

\usepackage[latin1]{inputenc}

\usepackage[dvipsnames]{xcolor}

\usepackage[german, english]{babel}
\usepackage{hyperref}
\usepackage{enumitem}

\newcommand{\be}{\begin}
\newcommand{\e}{\end}
\newcommand{\beq}{\begin{equation}}
\newcommand{\eeq}{\end{equation}}

\newtheorem{lemma}{Lemma}[section]

\newcommand{\om}{\omega}
\newcommand{\Om}{\Omega}

\theoremstyle{definition}

\newtheorem{remark}[lemma]{Remark}

\numberwithin{equation}{section}
\newcommand{\comment}[1]{}

\newcommand{\curly}[1]{\mathcal{#1}}
\newcommand{\setof}[2]{\left\{ #1\; : \;#2 \right\}}

\newcommand{\Z}{{\mathbb Z}}
\newcommand{\R}{{\mathbb R}}

\newcommand{\OO}{{\mathcal{O}}}

\newcommand{\La}{{L}}
\newcommand{\q}{{Q}}

\newcommand{\Lp}{{\Delta}}

\newcommand{\supp}{{\mathrm {supp}\,}}

\newcommand{\al}{{\alpha}}
\newcommand{\de}{{\delta}}
\newcommand{\eps}{{\varepsilon}}

\newcommand{\ph}{{\varphi}}
\newcommand{\lam}{{\lambda}}
\newcommand{\gam}{{\gamma}}

\renewcommand{\l}{\left}
\renewcommand{\r}{\right}

\newcommand{\scp}[2]{\langle#1,#2\rangle}

\newcommand{\qmexp}[1]{\left\langle #1\right\rangle}

\newcommand{\Hm}[1]{\leavevmode{\marginpar{\tiny%
$\hbox to 0mm{\hspace*{-0.5mm}$\leftarrow$\hss}%
\vcenter{\vrule depth 0.1mm height 0.1mm width \the\marginparwidth}%
\hbox to 0mm{\hss$\rightarrow$\hspace*{-0.5mm}}$\\\relax\raggedright
#1}}}

\begin{document}

\title[Optimal Hardy weights on the Euclidean lattice]{Optimal Hardy weights on the\\ Euclidean lattice}

\begin{abstract} We investigate the large-distance asymptotics of optimal Hardy weights on $\mathbb Z^d$, $d\geq 3$,  via the super solution construction. For the free discrete Laplacian, the Hardy weight asymptotic is the familiar $\frac{(d-2)^2}{4}|x|^{-2}$ as $|x|\to\infty$. We prove that the inverse-square behavior of the optimal Hardy weight is robust for general elliptic coefficients on $\mathbb Z^d$: (1) averages over large sectors have inverse-square scaling, (2), for ergodic coefficients, there is a pointwise inverse-square upper bound on moments, and (3), for i.i.d.\ coefficients, there is a matching inverse-square lower bound on moments. The results imply $|x|^{-4}$-scaling for Rellich weights on $\mathbb Z^d$. Analogous results are also new in the continuum setting. The proofs leverage Green's function estimates rooted in homogenization theory.  
\end{abstract}

\author[M. Keller]{Matthias Keller}
\address{Matthias Keller, Universit\"at Potsdam, Institut f\"ur Mathematik, 14476  Potsdam, Germany}
\email{matthias.keller@uni-potsdam.de}
\author[M. Lemm]{Marius Lemm}
\address{Marius Lemm, Institute of Mathematics, EPFL, 1015 Lausanne, Switzerland}
\email{marius.lemm@epfl.ch}
\date{August 23, 2021}
\maketitle

\section{Introduction}
Since their origin in the early $20$th century, the study of Hardy inequalities has been a vigorous field of research.  
On the one hand, this stems from the great variety of applications Hardy inequalities have in analysis and mathematical physics, for example as uncertainty principles in quantum mechanics, \cite{Fr11}, for the solvability and growth control of differential equations \cite{AGG06} and in spectral graph theory \cite{Nag04}. On the other hand, they provide intriguing examples of functional inequalities where explicit sharp constants and asymptotics of minimizers or ground states can be studied. While most of the literature focuses on  the continuum setting of differential operators, the inequality was originally phrased and proven by Hardy in the discrete setting; see \cite{KMP06} for a review of the history. 

Classically, the Hardy weight $ w $ is expressed as the  inverse of the distance function to some power which is for example given by $ w(x)=\frac{(d-2)^2}{4}|x|^{-2} $ for the most familiar setting in $ \R^{d} $, $ d\ge 3 $. This weight is optimal in a precise sense and in particular the constant $ \tfrac{(d-2)^2 }{4}$ is sharp. In 2014, Devyver, Fraas, and Pinchover \cite{DFP} (see also \cite{DP}) presented a method---the so-called supersolution construction---to construct \emph{optimal} Hardy weights for general (not necessarily symmetric) positive second-order elliptic operators   on non-compact Riemannian manifolds. This method was later extended in \cite{KePiPo2} to weighted graphs where it was also shown that for the standard Laplacian on $ \Z^{d} $, $ d\ge 3 $, there is an optimal Hardy weight $ w $ which asymptotically satisfies the continuum asymptotics $ w(x)\sim \frac{(d-2)^2}{4}|x|^{-2}  $, improving the constant previously obtained by Kapitanski and Laptev \cite{KaL}. While the method using super solutions such as the Green's function to obtain Hardy inequalities and the ground state transform  seem  to be folklore, see e.g. \cite{Fi00,FS08,FSW08,Gol14}, the novelity of the approach lies in the use of solutions with specific properties and the proof of optimality. For further recent work on optimal decay of Hardy weights in the continuum, see \cite{BGGP,DPP,KP,PV}.

\subsection{Summary of main results about optimal Hardy weights}
In this paper we establish inverse-square behavior at large distances of optimal Hardy weights for elliptic operators on $ \Z^{d} $ with $d\geq 3$. Specifically, we prove \textit{inverse-square bounds on optimal Hardy weights} in the following settings.
\begin{itemize}
	\item For general coefficients, upper and lower bounds on sectorial averages (Theorem~\ref{thm:mainspatial}). 
	\item For ergodic coefficients satisfying a logarithmic Sobolev inequality (for example independent and identically distributed random variables), pointwise upper bounds on moments (Theorem~\ref{thm:mainrandom}).
	\item For independent and identically distributed random coefficients, pointwise lower bounds on moments (Theorem \ref{thm:iid}).
\end{itemize}

An upper bound on an optimal Hardy weight informs on what is the best possible expected Hardy inequality \eqref{eq:hardy}, while a lower bound corresponds to a concrete Hardy inequality that is useful in applications.

We comment on the fact that all of the bounds we prove here involve some kind of averaging or at least exclude a probability-zero set. One may naively hope that, for fixed ellipticity ratio, one can prove inverse-square bounds on optimal Hardy weights that hold (a) pointwise and (b) deterministically (i.e., uniformly in the coefficients). However, results from classical elliptic regularity theory \cite{A1,A2,deGiorgi,LSW63,Moser,Nash} show this is impossible and so one needs to compromise on either (a) or (b). More precisely, optimal Hardy weights can be expressed through Green's function derivatives.  For the latter, in the context of elliptic regularity it is well-known that pointwise, deterministic bounds of the required scaling do not hold; see Section \ref{sect:ERT} for more on this.

Here, we address this issue by proving suitably averaged versions of the widely expected inverse-square behavior of optimal Hardy weights on $\Z^d$.

On the one hand, one can compromise on the pointwise nature of the bound by spatial averaging; this leads us to Theorem~\ref{thm:mainspatial}. On the other hand, one expects that the counterexamples developed in elliptic regularity are non-generic in an appropriate sense. This motivates the study of  random coefficients, which physically model materials with disordered microstructure. The associated regularity theory belongs to the extremely active area of stochastic homogenization theory \cite{ABM,AKMpaper,AKM,AS,BG,BGO1,BGO2,B,CKS,CN,DG,DD,DGL,DGO,GNO,GO,GM,KL,MN,MO1,MO2,MO3} and we draw heavily on some relatively recent precise descriptions of Green's function derivatives, specifically \cite[Theorem 8.20 \& Section 9.2]{AKM}, \cite[Corollary 3]{BGO1}, or \cite[Proposition 4.2]{GM}. Along the way, we find that the leading-order asymptotic behavior of the optimal Hardy weight depends on the usual ``correctors'' and will thus remain genuinely random, cf.\ \cite{AKM,MN,MO3}.



Our main focus lies on the treatment of the discrete setting since it is more interesting to us from a technical point of view. Nonetheless, we mention that analogous results in the continuum case do not seem to have appeared previously but can be derived in a parallel manner. We include an example in Section \ref{sect:continuum}.

%
%

 \pagebreak

\subsection{Organization of the paper}
The paper is organized as follows.
\be{itemize}

\item In \textbf{Section \ref{sect:prelim}}, we recall the elliptic graph setting, review the definition of optimal Hardy weights and their graph-intrinsic supersolution construction from \cite{KePiPo2}. 

\item In \textbf{Sections \ref{sect:main1}} we state our main results on optimal Hardy weights with spatial averaging.

\item In \textbf{Sections \ref{sect:main2}} we state our main results on optimal Hardy weights with probabilistic averaging. 

\item The brief \textbf{Section \ref{sect:general}} contains rather general upper and lower bounds relating optimal Hardy weights and Green's functions on any elliptic graph.
 
\item In \textbf{Section \ref{sect:mainspatialpf}},  we prove the results of \textbf{Section \ref{sect:main1}} via spatial averaging.
\item In \textbf{Sections \ref{sect:mainrandompf} and \ref{sect:iidpf}}, we prove the main results of \textbf{Section \ref{sect:main2}} via probabilistic averaging.
\item In the short \textbf{Section \ref{sect:rellich}}, we apply the results to Rellich inequalities on elliptic graphs. Specifically, we use the techniques of \cite{KePiPo4} to derive the expected $|x|^{-4}$ scaling in a probabilistic sense.

\item The \textbf{Appendix} contains the proof of a technical Lemma \ref{sect:onedirectionpf} and the proof of an elementary fact about the free Green's function that does not seem to appear in the standard literature --- it is nowhere locally constant.
\e{itemize}

An open problem is described in Remark \ref{rmk:iid} (iii).\\

We note that an earlier preprint version of this manuscript also contained results on the asymptotic expansion of the averaged Green's function and its derivatives. These results have been moved into the separate paper \cite{KLother}.


\section{Preliminaries}\label{sect:prelim}
\subsection{The elliptic graph setting}
Let $ X $ be a discrete countable set and let $b $ be a  \emph{graph} over $ X $. That is $ b:X\times X\to[0,\infty) $ is a symmetric map with zero diagonal. We denote $ x\sim y $
whenever $ b(x,y)>0 $. The graph  $ b $ is assumed  to be \emph{connected} that is if for any to vertices $ x,y\in X $ there is a path $ x=x_{0}\sim \ldots\sim x_{n}=y $.
We denote
by $ \deg $ the \emph{weighted vertex degree}
\begin{align*}
\deg(x)=\sum_{y\in X}b(x,y).
\end{align*}

We assume that the graph satisfies the following ellipticity condition for some $ E>0 $
\begin{align}\label{E} 
 b(x,y)\ge E\sum_{z\in X}b(x,z)
\end{align}
for all $ x\sim y $.
It is not hard to see that \eqref{E} implies that the graph has bounded combinatorial degree, i.e., 
\begin{align*}
	 \sup_{x\in X}\#\{y\in X\mid b(x,y)>0\} <\infty.
\end{align*}
 Hence, $ \deg $ takes  finite values and the graph is locally finite.

Let $ \La $ be the operator acting on $ C(X) $, the real valued functions on $ X $, via
\beq\label{eq:Ldefn}
\La f(x)=\sum_{\substack{y\in X}}b(x,y)(f(x)-f(y)).
\eeq
Restricting $ \La $ to the compactly supported functions $ C_{c}(X) $, we obtain a symmetric operator on $ \ell^{2}(X) $ whose Friedrich extension  is with slight abuse of notation also  denoted by $ L $ with domain $ D(L) $.
Note that $ L $ is bounded if and only if $ \deg $ is bounded and in this case $ D(L)=\ell^{2}(X) $, see e.g. \cite[Theorem~9.3]{HKLW}. The associated quadratic form 
$$
\q(f)=  \sum_{x, y\in X}b(x,y)(f(x)-f(y))^{2}
$$
is the Dirichlet energy of $f\in C(X)$ and for $ f\in D(L) $ we have
\begin{align*}
\q(f)=2\langle f,Lf\rangle.
\end{align*}
Denote by $ G: X\times X\to[0, \infty] $ the Green function of $L$ which is given by
\begin{align*}
G(x,y)=\lim_{\al\searrow 0}(L+\al)^{-1}1_{x}(y),
\end{align*}
where $ 1_{x} $ is the characteristic function of the vertex $ x $.
In case $ G(x,y) $ is finite for some $ x,y\in X $ it is finite for all $ x,y \in X $ due to connectedness of the graph. In this case the graph is called \emph{transient}. We say the Green's function is \emph{proper} if $ G(o,\cdot) $ is a proper function for some vertex $ o\in X $. (We recall that a function is proper if the preimages of compact sets are compact.) 

In the following, whenever we consider elliptic operators on $ \Z^{d} $, we will always have that the Green function $ G(o,\cdot) $ is proper. Also, we will often denote $$ G(x)=G(o,x) $$
and choose $ o =0$ whenever $ X=\Z^{d} $.
\subsection{Optimal Hardy weights}
In this paper, we are interested in the large-distance behavior of optimal Hardy weights. 

We say $ w:X\to [0,\infty) $ is a \emph{Hardy weight} for the graph $ b $ if all $ f\in C_{c}(X)$ satisfy the \emph{Hardy inequality} with weight $w$,
\beq\label{eq:hardy} 
\begin{aligned}
\q(f)=\sum_{x, y\in X}b(x,y)(f(x)-f(y))^{2}\ge \sum_{x\in X}w(x)f(x)^{2}.
\end{aligned}
\eeq
This inequality can be extended to the extended Dirichlet space, i.e., the closure with respect to $ \q^{1/2} $ which becomes a norm as we assumed that the graph is transient, see e.g. \cite{FOT} for general Dirichlet forms or \cite[Theorem~B.2]{KLSW} in the graph case. In particular, the constant function $ 1 $ is not in the extended Dirichlet space of transient Dirichlet forms. If $ \deg $ is bounded, then $ L $ is bounded and $ \ell^{2}(X) $ is included in the extended Dirichlet space, see \cite{HKLW}.

A \emph{ground state} for an operator $ H $ is a positive non-trivial solution $ u\in C(X) $, $ u\ge 0 $ to $ Hu=0 $ such that for any other super-solution $ v \ge 0$ to $ Hv\ge0 $ outside of a compact set with $ u\ge v\ge 0 $  there exists $ C\ge0 $ such that $ v\ge Cu $. In our situation of locally finite graphs a positive solution to $ Hu=0 $ always exist by the virtue of the Allegretto-Piepenbrink theorem, \cite[Theorem~3.1]{HK} and $ u $ is a ground state in the critical case (for the definition of criticality see below).

The following definition of optimal Hardy weights was first made by Devyver, Fraas, and Pinchover \cite{DFP}, see also \cite{DP}, for elliptic operators in the continuum and it was later investigated in \cite{KePiPo2} for graphs.
\pagebreak
\be{definition}\label{defn:optimal}
We say that $w$ is an \textit{optimal Hardy weight} if it satisfies the following three conditions.
\be{itemize}
	\item[(a)]	\textit{Criticality}: For every $\tilde w\geq w$ but $\tilde w\neq w$, the Hardy inequality fails.
	\item[(b)] \textit{Null-criticality:} The formal Schr{\"o}dinger operator $\La-w$ does not have a ground state in $\ell^2(X)$. 
	\item[(c)] \textit{Optimality near infinity:} For any $\lambda>0$ and any compact set $K$, the function $(1+\lambda)w$ fails to be a Hardy weight for compactly supported functions in $X\setminus K$.
	 \e{itemize}
\e{definition}

We note that (a) and (b) imply (c) as can be seen from the proof in \cite{KePiPo2} and it was shown in the continuum in \cite{KP}. Another answer to how large  Hardy weight can be is given in \cite{KP} in the continuum case in terms of integrability in at infinity.

\subsection{Connection to the Green's function}
The results of \cite{KePiPo2} (building on \cite{DFP,DP}) derived optimal Hardy weights in an intrinsic way via superharmonic functions. More precisely, by \cite[Theorem~0.1]{KePiPo2}, the function $ w:X\to[0,\infty) $ defined by
	\begin{align*}
	w(x)=2\frac{\La u^{\frac{1}{2}}(x)}{u^{\frac{1}{2}}(x)}
	\end{align*}
	is an optimal Hardy weight whenever $ u $ is a proper positive superharmonic function which is harmonic outside of a finite set and satisfies 	
	the growth condition $ \sup_{x\sim y} u(x)/u(y)<\infty$. Now, $ G(o,\cdot)=G(\cdot) $ is a  positive superharmonic function which  is  harmonic outside of $ \{o\} $. Then, \cite[Theorem~0.2]{KePiPo2} implies that
	\begin{align}\label{eq:wG}
	w_G(x)=2\frac{\La G(x)^{\frac{1}{2}}}{G(x)^{\frac{1}{2}}}=1_{o}(x)+\frac{1}{G(x)}
	\sum_{\substack{y\in X: \\y\sim x}}b(x,y)(G(x)^{\frac{1}{2}}-G(y)^{\frac{1}{2}})^{2}
	\end{align}	
	is an optimal Hardy weight. The condition  $ \sup_{x\sim y} G(x)/G(y)<\infty$ is satisfied by  \eqref{E}, see Lemma~\ref{lm:basic} below, and, therefore,  $ w_{G} $ is optimal in the sense of Definition~\ref{defn:optimal} whenever $ G $ is proper.

We will make use of the facts discussed in this section as follows:

\be{enumerate}
	\item We can use bounds on the Green's function and its derivatives to obtain information on $w_G$ via \eqref{eq:wG}.
	\item We can use the optimality of the Hardy weight $w_G$ to conclude asymptotic bounds on all possible Hardy weights $w$.
\e{enumerate}

In previous works \cite{KePiPo2,KePiPo3}, this has been implemented for special operators like the free Laplacian on $\mathbb Z^d$ where the Green's function is, at least asymptotically, exactly computable through the Fourier transform. Here we show for the first time that the Hardy weight scaling of $|x|^{-2}$ is robust in various senses for a wide class of elliptic divergence-form operators on $\Z^d$. For us, the Fourier transform is no longer available and other methods have to be used.

\subsection{From upper bounds on $w_G$ to asymptotic bounds on all Hardy weights}
In this short section, we make precise Step (2) mentioned above, i.e., how to go from upper bounds on $w_G$ to asymptotic bounds on all Hardy weights. Note that the Hardy inequality \eqref{eq:hardy} is stronger if the weight $w$ is large, so we are interested in limiting how large a general Hardy weight can be, given a bound on the special optimal weight $w_G$.

Let $ d $ be some metric on $ X $ and $ \|x\|=d(x,o) $ for $ x\in X $ and $ o\in X $ which was used in the definition of $ G(x)=G(o,x) $. Below when we consider $ X=\Z^{d} $ we will use the Euclidean norm $ |\cdot| $ for $ \|\cdot\| $
and $ o=0 $.

\be{proposition}\label{prop:wGw} Assume that $b$ is a connected transient graph over $X$ which satisfies \eqref{E} and $ G $ is proper. 
Assume that there exists a constant $C>0$ and a power $p>0$ such that 
\beq\label{eq:wGrate}
w_G(x)\leq \frac{C}{1+\|x\|^{p}},\qquad  x\in X.
\eeq
Suppose a function $\tilde w:X\to (0,\infty)$ satisfies
$$
\tilde w(x)\geq (1+\lam)\frac{C}{1+\|x\|^p},\qquad  x\in X\setminus K,
$$
for some $\lam>0$ and some compact set $K\subseteq X$. Then $\tilde w$ is not a Hardy weight.
\e{proposition}

In particular, Proposition \ref{prop:wGw} shows that the asymptotic decay rate of $w_G$ controls the asymptotics of all Hardy weights: Assuming we know that \eqref{eq:wGrate} holds, there cannot exist a Hardy weight $\tilde w$ satisfying $\tilde w(x)\gtrsim \|x\|^{\eps-p}$ asymptotically (meaning outside of compacts), for any $\eps>0$. Here, $\|x\|^\eps$ can be replaced by any function that goes to infinity as $\|x\|\to\infty$, say $\log\log\|x\|$.

\be{proof}[Proof of Proposition \ref{prop:wGw}]
By optimality at infinity of $w_G$, we know that or every $\lam>0$ and every compact set $K$, there exists a function $f$ supported in $X\setminus K$ such that the Hardy inequality with weight $(1+\lam)w_G$ fails. By the assumed bound on $w_G$, this implies that the Hardy inequality also fails for any $\tilde w$ such that
$$
\tilde w(x)\geq \frac{C(1+\lam)}{\|x\|^{p}}\geq (1+\lam)w_G(x),\qquad \textnormal{ on } X\setminus K.
$$
This proves Proposition \ref{prop:wGw}.
\e{proof}

We remark that Proposition \ref{prop:wGw} only uses optimality near infinity of $w_G$, i.e., property (c) of Definition~\ref{defn:optimal}. Properties (a) and (b) also transfer information from optimal Hardy weights to general Hardy weights, but we focus on property (c) because it concerns the asymptotic properties of Hardy weights and is therefore most useful to us.

\section{Bounds on optimal Hardy weights with spatial averaging}\label{sect:main1}
In this section, we discuss our first class of main results, namely upper and lower bounds on the optimal Hardy weight $w_G$ for models on $\Z^d$, $d\geq 3$. These bounds translate to information on the best possible large-distance decay of any Hardy weight via Proposition~\ref{prop:wGw}.

\subsection{The benchmark: the free Laplacian on $\Z^d$}
We are interested in the long-distance behavior of the optimal Hardy weight $w_G$ from \eqref{eq:wG}. To understand this, we require asymptotic control on $G$ and $\nabla G$ (for a definition of $ \nabla G $ see the next subsection) which is made explicit Proposition \ref{prop:wGbounds} below. 

The benchmark is the case of the free Laplacian $ \Delta $ on $\Z^d$, $d\geq 3$, 
which is given via the standard edge weights $ b(x,y)=1 $ if $ |x-y|=1 $ and zero otherwise. The Green's function $ G_{0} $ of $ \Delta $ is asymptotically explicitly computable and satisfies  $G_0(x)\sim |x|^{2-d}$ and $|\nabla G_0(x)|\sim |x|^{1-d}$; see e.g.\ \cite{Spitzer}, where $$  
|x|=\sqrt{|x_1|^{2}+\ldots+|x_d|^{2}}
 $$ denotes the Euclidian norm. Using asymptotic formulas, the optimal Hardy weight, including the sharp constant, is computed as
\begin{align}\label{eq:wG0}
w_{G_0}(x)={2}\frac{\Lp[G_0^{1/2}(x)]}{G_0^{1/2}(x)}=\frac{\big( d-2 \big)^2}{4} \,\frac{1}{|x|^2} + \OO\Bigg( \frac{1}{|x|^3} \Bigg), \quad\,\, x\in \Z^d,
\end{align}
in \cite[Theorem~8.1]{KePiPo2}.
See \cite{KaL} for a similar estimate with a different constant.

Our goal in this paper is to show that the inverse square decay is robust, in various senses, within the class of elliptic graphs on $\Z^d$, $d\geq 3$,whose Green's functions are no longer known to be explicitly computable. 

\subsection{Elliptic operators on $\Z^d$}
Let us introduce the setting of elliptic operators on $\Z^d$,  $d\geq 3$.  Let $\mathbb E^d$ denote the set of undirected edges $[x,y]$ with nearest neighbors $x,y\in \Z^d$, i.e. $ |x-y|=1 $. For  $\lam\in (0,1]$, we define the class of uniformly elliptic coefficient fields
\begin{align*}
\mathcal{A}_{\lam}:=\setof{a:\mathbb E^d\to \R}{\lam \leq a\leq 1}.	
\end{align*}
We will study the following class of examples. For $a\in \mathcal{A}_{\lam}$, we set
\beq\label{eq:setting2}
X=\Z^d,
\qquad b(x,x+e_j)=b(x+e_j,x)=a([x,x+e_j])
\eeq
for all $ x\in\Z^{d} $ and the standard basis $ e_{j} $, $1\leq j\leq d$. One can easily verify that the graph defined by \eqref{eq:setting2} satisfies the ellipticity condition \eqref{E} for some number $E=E(d,\lam)>0$. 
We define the associated operator $\La$ via \eqref{eq:Ldefn} and also denote by $ L $ corresponding self-adjoint operator on $ \ell^{2}(\Z^{d}) $. Equivalently, we may express this as $$ \La=\nabla^* a \nabla $$ for appropriately defined $\nabla,\nabla^*$,  cf. \cite{MO1},
 i.e.,
\begin{align*}
	\nabla :C(\Z^{d})\to C(\Z^{d})^{d}, \quad 
	\nabla f(x )=
	\left(\begin{matrix}
		f(x+e_{1})-f(x)\\
		\vdots\\
		f(x+e_{d})-f(x)
	\end{matrix}\right),
\end{align*}
$ a $ is identified with the operator induced by the  $ d\times d $ diagonal matrix
\begin{align*}
	\left(\begin{matrix}
		a([\cdot,\cdot+e_{1}])&&0\\
	&	\ddots&\\
0	&&a([\cdot,\cdot +e_{d}])
	\end{matrix}\right)
\end{align*}
and $ \nabla ^{*} $ is the formal adjoint of $\nabla  $ which acts as
\begin{align*}
	\nabla^{*} :C(\Z^{d})^{d}\to C(\Z^{d}), \quad \nabla ^{*}
\left(\begin{matrix}
	f_1(x)\\
	\vdots \\
	f_d(x)
\end{matrix}\right)= \sum_{j=1}^{d}(f_j(x-e_{j} )-f_j(x)).
\end{align*}
With slight abuse of notation we also write for $ f\in C(\Z^{d}) $ an edge $ e=[x,x+e_{j}] $
\begin{align*}
	\nabla f(e)=f(x+e_{j})-f(x).
\end{align*}

\subsection{Background: Elliptic regularity theory and the Green's function}\label{sect:ERT}
In order to correctly estimate the optimal Hardy weights of these graphs defined by \eqref{eq:wG}, the ideal situation would be to  have upper and lower bounds on their Green's functions which agree with the behavior $G_0(x)\sim |x|^{2-d}$ and $|\nabla G_0(x)|\sim |x|^{1-d}$ of the free Green's function $G_0$ that we reviewed above. Deriving such bounds on $\nabla G$ and its continuum analog are old problems in potential theory and probability theory with close connections to the elliptic regularity theory developed by de Giorgi, Nash, and Moser in the 1950s \cite{deGiorgi,Moser,Nash}.

On the one hand, the situation for $G$ itself is relatively pleasant. Based on Nash's parabolic comparison estimates to the free Laplacian, Littman, Stampacchia, and Weinberger \cite{LSW63} and Aronson \cite{A1,A2} proved that the $|x|^{2-d}$-decay is preserved uniformly in the class $\curly{A}_{\lam}$, i.e., there exists a constant $C_{d,\lam}>1$ such that
\beq\label{eq:aronson}
C^{-1}_{d,\lam} |x-y|^{2-d}\leq G(x,y)\leq C_{d,\lam} |x-y|^{2-d},\qquad  x, y\in \R^d, x\neq y.
\eeq
(This is commonly known as the Nash-Aronson bound.)
We emphasize that $C_{d,\lam}$ depends on the coefficient field $a$ only through the ellipticity parameter $\lam$. Note that \eqref{eq:aronson} concerns the continuum setting. For the extension to the discrete case, see \cite{DD,SC} and references therein. In the discrete case, we may use the fact that the Green's function is bounded on the diagonal to write
\beq\label{eq:aronsondiscrete}
C^{-1}_{d,\lam} (1+|x-y|)^{2-d}\leq G(x,y)\leq C_{d,\lam} (1+|x-y|)^{2-d},\qquad  x,y\in\Z^d.
\eeq
We remark that the bounds in \eqref{eq:aronson} and \eqref{eq:aronsondiscrete} establish that the Green's function is proper. We capture this fact because of its relevance in the following lemma.

\begin{lemma}Whenever, $ a\in \mathcal{A}_{\lambda} $ for $ \lambda \in(0,1]$, the  Green's function $G(y,\cdot)$ is proper for any fixed $y\in \Z^d$.
\end{lemma}

It turns out that bounds on $\nabla G$  are a more subtle and delicate matter for general elliptic equations (discrete or continuous). Indeed, the standard elliptic regularity theory \cite{deGiorgi,Nash,Moser} only yields
$
|\nabla G(x)|\leq C(d,\lam) |x|^{2-d-\al(d,\lam)}
$
for some $\al(d,\lam)>0$; see also \cite{MO1,MO2}. Moreover, within the general elliptic framework as defined above, there are counterexamples in the continuum that show one cannot expect a pointwise upper bound $|\nabla G(x)|\leq C|x|^{1-d}$ to hold with an $a$-uniform constant $C$. Indeed, explicit counterexamples can be constructed using the theory of quasi-conformal mappings; see p.\ 795 in \cite{GO} and p.\ 299 in \cite{GT}.

What does this mean for the optimal Hardy weight $w_G$? \textit{In summary, we cannot hope to prove $a$-uniform bounds on optimal Hardy weights that simultaneously hold (a) pointwise and (b) deterministically.} In the following sections, we show that after loosening either requirement (a) or (b), inverse-square bounds do hold.

\subsection{Inverse-square bounds on spatial averages}
We follow the order laid out in the introduction and begin with the most general results which hold for all coefficient fields $a\in \curly{A}_{\lam}$. As we reviewed just above, this requires compromising on the pointwise nature of the bound and so we consider spatial averages.

We introduce the cone $\curly C_{j,\al}\subseteq \R^d$ in direction $j\in \{1,\ldots,d\}$ with opening angle given in terms of $\al\in (0,1)$, i.e.,
$$
\curly C_{j, \al}=\setof{x\in \R^d}{\scp{x}{e_j}>(1-\al)|x|}.
$$
We then define the associated discrete sector of radial size $\ell>0$ on scale $R>0$ by
$$
\curly S^{j,\al}_{R,\ell}=\setof{x\in\Z^d\cap \curly C_{j,\al}}{R\leq |x|\leq \ell R}.
$$

\be{theorem}[Inverse-square  bounds on sectorial averages]\label{thm:mainspatial}
Let $d\geq 3$ and $\lam\in (0,1]$. There exist constants $c_{d,\lambda}>1$  and  $\ell=\ell_d>1$ such that for every $a\in\curly{A}_{\lam}$ and every choice of $R\geq 1$, $j\in \{1,\ldots,d\}$ and $\al\in (0,1)$, 
\beq\label{eq:mainspatial}
  {c_{d,\lam}}^{-1} R^{-2}
 \leq
  R^{-d}\sum_{x\in \curly S^{j,\al}_{R,\ell_d}} w_G(x)\leq c_{d,\lam} R^{-2}.
\eeq
\e{theorem}

Note that the cardinality of $\curly S^{j,\al}_{R,\ell_d}$ is of order $R^d$, so up to a change in the constant $c_{d,\lam}$, \eqref{eq:mainspatial} indeed gives bounds on \textit{averages over large sectors}. The sectors can be replaced by other sufficiently large sets, e.g., annuli.

Theorem \ref{thm:mainspatial} is proved in Section~\ref{sect:mainspatialpf}.

\subsection{Optimal Hardy weights in the continuum}\label{sect:continuum}
While we focus on the discrete case, in this section we also discuss the continuum. Optimal Hardy weights were originally considered in the continuum in work of Devyver, Fraas, and Pinchover, \cite{DFP,DP}. The notion of optimality is analogous to Definition~\ref{defn:optimal}; see   \cite[Definition~2.1]{DFP}. 

We consider an elliptic, divergence form operator
$$
P=-\mathrm{div}(A\nabla)
$$
on $\R^d$ with a scalar function $A$ belonging to
$$
\curly{A}^{cont}_\lam=\setof{\tilde A:\R^d\to\R}{ \textnormal{ H\"older continuous, }\lam\leq \tilde A\leq 1 \textnormal{ a.e.}}
$$
 for some $\lam>0$. Moreover, \cite{DFP} mention that the assumption of H\"older continuity of $\tilde A$ can be relaxed to any assumption guaranteeing standard elliptic regularity theory, which in addition to $\lam\leq \tilde A\leq 1$ means just measurability \cite{deGiorgi,Moser,Nash}.
  A Hardy weight is defined as a non-zero function $W:\R^d\setminus\{0\}\to[0,\infty)$ verifying 
$$
\scp{\varphi}{P\varphi}\geq \int_{\R^d\setminus\{0\}} W(x) |\varphi(x)|^2 \mathrm{d} x,\qquad  \varphi \in C_0^\infty(\R^d\setminus\{0\}).
$$
According to Theorem 2.2 in \cite{DFP}, an optimal Hardy weight can be defined in terms of the minimal positive Green's function with a pole at the origin, $G$, in the following way
\beq\label{eq:WG}
W_G(x)=\l|\nabla\log \sqrt{G(x)}\r|_A^2=\frac{|\nabla G(x)|_A^2}{4G(x)^2},\qquad \text{where } |\xi|_A^2=\xi\cdot A\xi.
\eeq

We have the following continuum analog of Theorem \ref{thm:mainspatial}. For  $j\in \{1,\ldots,d\}$, $\al\in (0,1)$, and $R,\ell>0$, we define the sector
$$
S^{j,\al}_{R,\ell}=\setof{x\in\R^d\cap \curly C_{j,\al}}{R\leq |x|\leq \ell R}.
$$

\be{proposition}[Inverse square behavior of annular averages --- continuum version]\label{prop:continuum}
Let $d\geq 3$ and $\lam\in(0,1]$. There exists constants $c_{d,\lam},\ell_d>1$ so that for every $A\in \curly{A}^{cont}_\lam$ and every choice of $R\geq 1$, $j\in \{1,\ldots,d\}$ and $\al\in (0,1)$, 
\begin{align*}
{c_{d,\lam}}^{-1} R^{-2}
 \leq
  R^{-d}\int\limits_{S^{j,\al}_{R,\ell_d}} W_G(x) \mathrm{d} x \leq c_{d,\lam} R^{-2}.	
\end{align*} 
\e{proposition}

The proof of Proposition \ref{prop:continuum} is a continuum analog of the proof of Theorem \ref{thm:mainspatial}; see Section \ref{sect:continuumpf} for a sketch.

For readers interested in the continuum situation, we mention that also the probabilistic results discussed in the next section have continuum analogs that appear to be new. Since we focus on the discrete setting here, we leave the extensions to the interested reader.

\section{Bounds on optimal Hardy weights for random coefficients}\label{sect:main2}
\subsection{Upper bound for ergodic coefficients}
We can hope to derive pointwise bounds if we compromise on the deterministic nature of the bound, essentially hoping that the counterexample from the elliptic regularity theory reviewed in Section \ref{sect:ERT} are appreciably non-generic and thus do not affect the generic behavior of the optimal Hardy weight.

To this end, we now introduce elliptic operator with ergodic coefficients. Physically, these model materials with disordered microstructure.
We now describe the mathematical setup. Let $\mathbb P$ be a probability measure on the measure space $\curly{A}_{\lam}$ endowed with the product topology. We write $\qmexp{\cdot}$ for the associated expectation value.

\be{assumption}\label{ass:P}
We make the following two assumptions on $\mathbb P$.
\be{enumerate}
\item $\mathbb P$ is stationary, i.e., for any $z\in\Z^d$, the translated coefficient field $a(\cdot+z)$ is also distributed according to $\mathbb P$.
\item $\mathbb P$ satisfies a logarithmic Sobolev inequality with constant $\rho>0$, i.e., for any random variable $\zeta:\curly{A}_{\lam}\to \R$, it holds that
\beq\label{eq:LSI}
\qmexp{\zeta^2 \log\frac{\zeta^2}{\qmexp{\zeta^2}}}\leq \frac{1}{2\rho}\qmexp{\sum_{e\in \mathbb E^d} \l(\mathrm{osc}_{e}\zeta\r)^2},
\eeq
where the oscillation of $\zeta$ is a new random variable defined by
$$
\l(\mathrm{osc}_{e}\zeta\r)(a)=\sup_{\substack{\tilde a\in \curly{A}_{\lam}:\\ \tilde a=a \mbox{\scriptsize { on }}\mathbb{E}^{d}\setminus\{e\}}} \zeta(\tilde a)-\inf_{\substack{\tilde a\in \curly{A}_{\lam}:\\  \tilde a=a \mbox{\scriptsize { on }}\mathbb{E}^{d}\setminus\{e\}}} \zeta(\tilde a).
$$
\e{enumerate}
\e{assumption}

In a nutshell, Assumption \ref{ass:P} specifies that the probability measure $\mathbb P$ is stationary and ``sufficiently'' ergodic. Indeed, the logarithmic Sobolev inequality implies a spectral gap and thus mixing time estimates for the associated Glauber dynamics \cite{GNO}. The idea that Assumption \ref{ass:P} is useful in the context of Green's function estimates in stochastic homogenization goes back to an unpublished manuscript of Naddaf-Spencer. We quote a well-known result which immediately provides lots of examples satisfying this assumption. 

\be{lemma}[The case of i.i.d.\ coefficients; see e.g.\ Lemma 1 in \cite{MO1}]
\label{lm:iid}
If the coefficients $\l(a(e)\r)_{e\in \mathbb E^d}$ are chosen in an independent and identically distributed (i.i.d.) way with values in $[\lam,1]$, then the induced measure $\mathbb P$ on $\curly{A}_{\lam}$ satisfies Assumption~\ref{ass:P} (with $\rho=1/8$).
\e{lemma}

\be{remark}
The logarithmic Sobolev inequality \eqref{eq:LSI} is a slightly weaker version of the more standard one in which the oscillation is replaced by the partial derivative $\frac{\partial \zeta}{\partial a(e)}$. The version \eqref{eq:LSI}  has the slight advantage that the above lemma holds for the most general choices of i.i.d.\ coefficient fields \cite{MO1}.
\e{remark}

We are now ready to state the main result in the probabilistic setting, which provides a pointwise upper bound on the random function $w_G:\Om\times\Z^d\to\R_+$ with probability $1$. Here $\Om$ denotes the underlying probability space. We will often suppress the $\om$-dependence of $w_G$.

\be{theorem}[Pointwise upper bounds]\label{thm:mainrandom}
Let $d\geq 3$, and suppose Assumption \ref{ass:P} holds. For every $p\geq 1$, there exists $C_{d,\lam,p}>0$ so that
\beq\label{eq:wGannealed}
\qmexp{w_G(x)^p}^{1/p}\leq C_{d,\lam,p} (1+|x|)^{-2}, \qquad x\in\Z^d.
\eeq
Moreover, for every $\eps>0$, there exists  finite sets $K_\eps(\om)\subseteq \Z^{d}$ so that
\beq\label{eq:wGprob1}
w_G(\om,x)\leq (1+|x|)^{-2+\eps},\qquad  \om\in \Om,\, x\in\Z^d\setminus K_{\eps}(\om)
\eeq
holds with probability $1$. 
\e{theorem}

Theorem \ref{thm:mainrandom} is proved in Section \ref{sect:mainrandompf}.

The first estimate \eqref{eq:wGannealed} is a pointwise bound on moments of the Hardy weight $w_G$. The second part of the statement holds pointwise with probability $1$, at the modest price of an $\eps$-loss in the exponent. The proof is given in Section~\ref{sect:mainrandompf}.\medskip


Since $w_G$ is an optimal Hardy weight, this theorem yields information on the best-possible decay of an arbitrary Hardy weight via Proposition~\ref{prop:wGw}. 

\be{corollary}[Pointwise bounds on all Hardy weights]\label{cor}
Let $d\geq 3$ and suppose Assumption \ref{ass:P} holds. Let $p>0$ and let $w:\Om\times \Z^d\to\R_+$ be a function so that for some $C(\om)>0$,
\beq\label{eq:cor}
w(\om,x)\geq C(\om) (1+|x|)^{-p},\qquad  x\in\Z^d
\eeq
holds with positive probability. If $w$ is a Hardy weight, then  $p\geq 2$.
\e{corollary}

\be{proof}
We prove the contrapositive statement. Suppose that \eqref{eq:cor} holds for some $p<2$ with positive probability. We have
$$
\mathbb P\l(\inf_{x\in\Z^d}w(x)(1+|x|)^p>0\r)>0.
$$
Let $\eps_0=\frac{2-p}{2}>0$ and apply \eqref{eq:wGprob1} of Theorem \ref{thm:mainrandom} to obtain
$$
\mathbb P\l(\inf_{x\in\Z^d}w_G(x)(1+|x|)^{2-\eps_0}<\infty\r)=1.
$$
On the positive-probability event where both estimates hold,  one can use $p<2-\eps_0$ and apply Proposition \ref{prop:wGw}, with $K$ taken to be an appropriate discrete ball, to conclude that $w$ is not a Hardy weight. This proves Corollary \ref{cor}.
\e{proof}

\subsection{Lower bounds for i.i.d.\ coefficients}\label{ssect:iid}
The inverse-square upper bounds on moments in Theorem \ref{thm:mainrandom} can be matched by lower bounds in the i.i.d.\ setting. In fact, one even obtains the (random) scaling limit of the optimal Hardy weight. 
We employ the framework of \cite{GM,MN,MO3} in which the i.i.d.\ random variables are functions of i.i.d.\ Gaussians. This is a purely technical assumption which simplifies some aspects compared to general i.i.d.\ random variables.

Fix $\lam\in (0,1]$. Let $F:\R\to [\lam,\lam^{-1}]$ be a function satisfying
$$
F\in C^2(\R),\qquad F,F''\in L^\infty(\R).
$$
Let $(\zeta_e)_{e\in \mathbb E^d}$ be a family of i.i.d.\ standard Gaussians. We define the i.i.d.\  coefficient field $a\in \mathcal A_\lam$ by
\beq\label{eq:aiid}
a(e)=F(\zeta_e).
\eeq

Our main result in this setting, Theorem \ref{thm:iid}, is the matching moment lower bound to Theorem \ref{thm:mainrandom}.


\be{theorem}[Pointwise lower bounds]\label{thm:iid}
\mbox{}
Let $d\geq 3$. Let $p\in (\frac{1}{2},\infty)$. There is a constant $C_{d,\lam,p}>0$ and a radius $R_{d,\lam,p}>0$ so that
\beq\label{eq:mainrandompert1}
\langle w_G(x)^p\rangle^{1/p} \geq C_{d,\lam,p} (1+|x|)^{-2}, \qquad  |x|\geq R_{d,\lam,p}.
\eeq
For $d\geq 5$, the bound holds over all of $\Z^d$. 
\e{theorem}

Theorem \ref{thm:iid} is proved in Section \ref{sect:iidpf}. 

For i.i.d.\ random coefficients, we can combine Theorems \ref{thm:mainrandom} and \ref{thm:iid} to conclude that, e.g., for $p\geq 1$, there exist constants $C_{d,\lam,p} ,C_{d,\lam,p}' >0$ so that
\begin{align*}
	C_{d,\lam,p} (1+|x|)^{-2}\leq \langle w_G(x)^p\rangle^{1/p} \leq C_{d,\lam,p}' (1+|x|)^{-2}
\end{align*}
for all sufficiently large $x$.

 We see that \textit{Theorems \ref{thm:mainrandom} and \ref{thm:iid} confirm the expected inverse-square behavior of the optimal Hardy weight $w_G(x)$ after averaging/in the sense of moments}.

\begin{remark}\label{rmk:iid}
(i) 
Note that in dimensions $d\geq 5$ Theorem \ref{thm:iid} establishes a \textit{global} lower bound on the Hardy weight, which is positive everywhere on $\Z^d$. A priori nothing stops the optimal Hardy weight $w_G:\Z^d\to [0,\infty)$ from having zeros. The strict positivity we prove here for $d\geq 5$ is the first that globally limits the vanishing of the optimal Hardy weight. This aspect relies on a non-standard but elementary fact about the free Green's function; see Appendix \ref{app:G0}. 

(ii)
Analogous bounds hold in the continuum for coefficients with a unit range of dependence. For this, one can use the results of \cite[Chapters 5 \& 9]{AKM}, in particular Theorem 5.32. Since our focus is on the discrete setting, these straightforward extensions are left to the interested reader.

(iii) We describe an open problem. The first two steps in the proof of Theorem~\ref{thm:iid} suggest that it may be possible to obtain the (genuinely random) asymptotics of $w_G(x)$ as $|x|\to\infty$, in the sense of identifying the scaling limit of their distribution after spatial averaging. More precisely, Lemmas \ref{lm:tildew} and \ref{lm:hatw}, show that for studying the leading-order asymptotics of $w_G(x)$ it suffices to consider $\hat w(x)$ defined in \eqref{eq:hatwdefn}. One may then hope to understand the asymptotics of $\hat w(x)$ by making use of results identifying the correlations of the Gaussian scaling limit of correctors in the i.i.d.\ setting \cite{AKM,GM,MN,MO3}. Considering the nature of the representation \eqref{eq:hatwdefn}, this would first require a precise study of the correlations between the random coefficients and products of correctors.
\end{remark}

Since we have control over moments of the pointwise optimal Hardy weight $w_G(x)$, the second moment method can be used to obtain pointwise lower bounds  with controlled positive probability.

\be{corollary}[Pointwise lower bounds with positive probability]\label{cor:smm}
\mbox{}
Let $d\geq 3$. There are constants $c_{d,\lam},C_{d,\lam}>0$ so that the following holds. 
\be{enumerate}
\item[(i)] For $d\in \{3,4\}$, there is a radius $R_{d,\lam}>0$ so that \begin{align*}	
\mathbb P\l(w_G(x)>C_{d,\lam} (1+|x|)^{-2}\r)\geq c_{d,\lam}>0, \qquad  |x|\geq R_{d,\lam}.
\end{align*}
\item[(ii)] Let  $d\geq 5$. Then
\begin{align*}	
\mathbb P\l(w_G(x)>C_{d,\lam} (1+|x|)^{-2}\r)\geq c_{d,\lam}>0,\qquad  x\in\Z^d.
\end{align*}
\e{enumerate}
\e{corollary}

\be{proof}[Proof of Corollary \ref{cor:smm}] This follows from the Paley-Zygmund inequality which says that for all $\theta\in(0,1)$ 
$$
\mathbb P\big(w_G(x)>\theta \qmexp{w_G(x)}\big)\geq (1-\theta)^2 \frac{\qmexp{w_G(x)}^2}{\qmexp{w_G(x)^2}}.
$$
We apply this with $\theta=\frac{1}{2}$ and use Theorem \ref{thm:iid} with $p\in \{1,2\}$.
\e{proof}

\section{General pointwise estimates}\label{sect:general}
In this short section, we summarize some general pointwise estimates that are used in more specific situations later and are collected here for convenience. We derive from the formula for $ w_{{G}} $ given by \eqref{eq:wG} upper and lower bounds which directly involve the discrete derivative $|G(x)-G(y)|$ for $x\sim y$ as summarized in the following proposition. These bounds hold for all elliptic graphs and without any averaging. 

Recall that $E>0$ is a constant so that \eqref{E} holds.

\be{proposition}[General pointwise bound]\label{prop:wGbounds}
Assume that $b$ is a connected transient graph over $X$ which satisfies \eqref{E}. Then,
\begin{align*}
w_G(x)
\leq& 1_{o}(x)+\frac{1}{G(x)^2}
	\sum_{\substack{y\in X}}b(x,y)(G(x)-G(y))^2,\\
	w_G(x)
\geq& 1_{o}(x)+(1+E^{-1/2})^{-1}\frac{1}{G(x)^2}
	\sum_{\substack{y\in X}}b(x,y)(G(x)-G(y))^2.
\end{align*}
\e{proposition}

We mention in passing that these bounds yield rough decay estimates on $w_G$ that do not require any information on $\nabla G$, though we will not use these rough bounds in the following. Indeed, from $\q(f)=2\scp{f}{Lf}$ and the defining property of the Green's function we have that
$$
\sum_{x\in X} w_G(x) G(x)^2 \leq C G(o) <\infty.
$$
which for example implies a weak decay estimate on $\Z^3$ via the Nash-Aronson bound \eqref{eq:aronsondiscrete}, (see also \cite{PV}).

The proof of Proposition \ref{prop:wGbounds} uses the following basic comparison property of the Green function of strongly elliptic graphs with bounded combinatorial degree.

\begin{lemma}\label{lm:basic}
Let $ b $ be a connected transient graph over $ X $  which satisfies the ellipticity condition \eqref{E}. Then, for all $ x\sim y $
	\begin{align*}
	G(x)\geq E  G(y).
	\end{align*}
\end{lemma}

\begin{proof}[Proof of Lemma \ref{lm:basic}]
Note that the Green's function is superharmonic, more specifically, it satisfies
\begin{align*}
\La G(o,\cdot)=1_{o}.
\end{align*}
 By the virtue of the ellipticity condition \eqref{E} we have for all $ x\sim y $
\begin{align*}
 G(x) \ge\frac{1}{\sum_{z\in X}b(x,z) } \sum_{z\in X}b(x,z) G(z)\ge E G(o,y)
\end{align*}
which proves Lemma \ref{lm:basic}.
\end{proof}

\begin{proof}[Proof of Proposition \ref{prop:wGbounds}]
For the upper bound, we note that
$$
|G(x)^{1/2}-G(y)^{1/2}|=\frac{|G(x)-G(y)|}{G(x)^{1/2}+G(y)^{1/2}}
\leq
\frac{|G(x)-G(y)|}{G(x)^{1/2}}.
$$
For the lower bound, we use that by Lemma \ref{lm:basic}, we have $G(y)\leq E^{-1}G(x)$ and so
$$
|G(x)^{1/2}-G(y)^{1/2}|=\frac{|G(x)-G(y)|}{G(x)^{1/2}+G(y)^{1/2}}	\geq (1+E^{-1/2})^{-1}\frac{|G(x)-G(y)|}{G(x)^{1/2}}.
$$
Squaring both sides and applying the resulting estimates to the formula \eqref{eq:wG} for $ w_{G} $  establishes Proposition \ref{prop:wGbounds}.	
\e{proof}

\section{Proof of inverse-square bounds with spatial averaging}
\label{sect:mainspatialpf}
\subsection{Proof of Theorem \ref{thm:mainspatial} }

\subsubsection{Proof of the lower bound in \eqref{eq:mainspatial}.}
Recall the definition of a sector 
$$
\curly{S}^{j,\al}_{R,\ell}=\setof{x\in\Z^d}{\scp{x}{e_j}>(1-\al)|x|,\;R\leq |x|\leq \ell R} 
$$
 for $ j\in\{1,\ldots,d\} $ and $\al, R,\ell>0 $.

By the lower bound in Proposition \ref{prop:wGbounds}, the upper Nash-Aronson bound \eqref{eq:aronsondiscrete} and $a\geq \lam$, we have
$$
\begin{aligned}
\sum_{\curly{S}^{j,\al}_{R,\ell}} w_G(x)
&\geq (1+E^{-1/2})^{-1}\sum_{x\in\curly{S}^{j,\al}_{R,\ell}} \frac{1}{G(x)^2}\sum_{\substack{i=1,\ldots,d\\ s=\pm1}}  a([x,x+se_j])|G(x)-G(x+se_i)|^2\\
& \geq (1+E^{-1/2})^{-1}C_{d,\lam}^{-2} \lambda (1+(\ell_dR)^{d-2})^{-2}	\sum_{\curly{S}^{j,\al}_{R,\ell}}\sum_{\substack{i=1,\ldots,d\\ s=\pm1}}  |G(x)-G(x+se_i)|^2\\
& \geq C'_{d,\lam} R^{2(2-d)}	\sum_{x\in\curly{S}^{j,\al}_{R,\ell}} |G(x)-G(x+e_j)|^2.
\end{aligned}
$$
The following lemma gives a general lower bound on sectorial averages of Green's function derivatives for all elliptic coefficient fields. 
 



\be{lemma}[Lower bound on sectorial averages]\label{lm:onedirection}
Let $d\geq 3$. There exist constants $\ell=\ell_d>0$ and $C_{d,\lambda}>1$ such that for every $a\in\curly{A}_{\lam}$, every $1\leq j\leq d$, every $ \al\in(0,1) $ and every radius $R\geq 1$, 
$$
R^{-d}\sum_{x\in \curly S^{j,\al}_{R,\ell}} |G(x)-G(x+e_j)|^2 \geq  C_{d,\lam} R^{2-2d}.
$$
\e{lemma}

Lemma \ref{lm:onedirection} is proved below. It allows us to conclude that
$$
R^{-d} \sum_{x\in\curly{S}^{j,\al}_{R,\ell}} w_G(x)
\geq C''_{d,\lam} R^{-2},	
$$
which is the lower bound in \eqref{eq:mainspatial}. To complete the proof, it remains to prove Lemma~\ref{lm:onedirection}.

\textit{Proof of Lemma~\ref{lm:onedirection}}
\label{sect:onedirectionpf}
Let $\al>0$. We consider the case $j=1$ without loss of generality. We let $\ell>0$ to be determined later. The idea of the proof is that by the Nash-Aronson bounds \eqref{eq:aronsondiscrete}, we know that $G$ must have decayed somewhat between the interior and the exterior boundary of the sector.

In a preliminary step, we replace the sector
$$
\curly{S}^{1,\al}_{R,\ell}=\setof{x\in\Z^d}{\scp{x}{e_1}>(1-\al)|x|,\;R\leq |x|\leq \ell R}
$$
by a cuboid. Observe that there are constants $c_{in}, c_{out},c_{orth}>0$ depending only on $d$, so that the cuboid
$$
\mathrm{Cub}=\setof{(x_1,\ldots,x_d)\in \Z^d}{c_{in} R\leq x_1\leq c_{out} \ell R \text{ and } |x_i|\leq c_{orth} \al R\mbox{ for all } i\ge 2}
$$
is contained in the sector $\curly{S}^{1,\al}_{R,\ell}$. Without loss of generality, we assume that $c_{in} R$ and $c_{out} \ell R$ are integers.  Note further that for any fixed $\ell$, $\mathrm{Cub}$ has cardinality of order $R^d$ as $R\to\infty$. We conclude that it suffices to prove the claim with $\mathrm{Cub}$ in place of $\curly{S}^{1,\al}_{R,\ell}$.

We introduce the inner and outer faces of the cuboid
$$
\begin{aligned}
F_{in}=&\setof{x=(c_{in} R,x_2,\ldots,x_d)}{|x_i|\leq c_{orth} \al R\,\mbox{ for all } i\ge 2},\\ 
F_{out}=&\setof{x=(c_{out} \ell R,x_2,\ldots,x_d)}{|x_i|\leq c_{orth} \al R\,\mbox{ for all }i\ge 2}.
\end{aligned}
$$

By the Nash-Aronson bounds \eqref{eq:aronsondiscrete}, we have
$$
\begin{aligned}
G(x)\geq& C_{d,\lambda}^{-1} (1+|x|)^{2-d}\geq C'_{d,\lambda} c_{in}^{2-d} R^{2-d},\qquad  x\in F_{in},\\
G(x)\leq& C_{d,\lambda} (1+|x|)^{2-d}\leq C_{d,\lambda} c_{out}^{2-d} (\ell R)^{2-d},\qquad  x\in F_{out}.
\end{aligned}
$$
Given $x\in F_{in}$, note that the point $x+R(c_{out}\ell-c_{in})e_1\in F_{out}$. By choosing $\ell=\ell_d>0$ sufficiently large, depending on $d$ but not on $R$, these bound imply that  the Greens' function differs by order $R^{2-d}$ between these two points. Namely, there exists  $C''_{d,\lambda} >0$ so that
\beq\label{eq:glb}
G(x)-G(x+Rc e_1)\geq C''_{d,\lambda} R^{2-d}, \qquad   x\in F_{in}.
\eeq
where we introduced $c:=c_{out}\ell_d-c_{in}$.

Since $F_{in}$ contains an order of $(c_{orth}\al R)^{d-1}\sim R^{d-1}$ many sites, this implies
$$
\sum_{x\in F_{in}} \l(G(x)-G(x+Rce_1)\r)\geq C'''_{d,\lambda} R.
$$

By telescoping (recall that $cR\in\Z_+$) and the Cauchy-Schwarz inequality, we can upper bound the left-hand side as follows
\beq\label{eq:telescopic}
\begin{aligned}
&\sum_{x\in F_{in}} \l(G(x)-G(x+Rce_1)\r)\\
&=\sum_{x\in F_{in}} \sum_{n=1}^{cR}\l(G(x+(n-1)e_1)-G(x+ ne_1)\r)\\
&\leq C_{d} R^{d/2}\sqrt{\sum_{x\in F_{in}} \sum_{n=1}^{cR} |G(x+(n-1)Re_1)-G(x+ ne_1)|^2}\\
&\leq C_{d} R^{d/2} \sqrt{\sum_{x\in \mathrm{Cub}} |G(x)-G(x+e_1)|^2}.
\end{aligned}
\eeq
We combine this with \eqref{eq:glb} to conclude
$$
C_{d,\lambda}'' R\leq  R^{d/2}\sqrt{\sum_{x\in \mathrm{Cub}}  |G(x)-G(x+e_1)|^2}.
$$
Since it suffices to prove the claim for $\mathrm{Cub}$, Lemma \ref{lm:onedirection} is proved.\qed

\subsubsection{Proof of the upper bound in \eqref{eq:mainspatial}.} We first claim the following auxiliary bound on the average over a dyadic annulus
\begin{align}
\label{eq:mainspatialub}
 R^{-d}\sum_{\substack{x\in\Z^d:\\ R\leq |x|\leq 2R}} w_G(x)
 \leq& c_{d,\lam} R^{-2}.
 \end{align}
Let $a\in\curly{A}_{\lam}$ and $R\geq 1$ be arbitrary and recall the setting given by \eqref{eq:setting2}. By the upper bound in Proposition \ref{prop:wGbounds}, the lower Nash-Aronson bound in \eqref{eq:aronsondiscrete} and $a\leq 1$, we have
\beq\label{eq:justaronson}	
\begin{aligned}
\sum_{\substack{x\in\Z^d:\\ R\leq |x|\leq 2R}} w_G(x)
&\leq \sum_{\substack{x\in\Z^d:\\ R\leq |x|\leq 2R}} \frac{1}{G(x)^2}\sum_{\substack{e\in\mathbb{Z}^d\\|e|=1}}
a([x,x+e])\l|G(x)-G(x+e)\r|^2\\
& \leq  C_{d,\lam}^2 R^{2(d-2)}	\sum_{\substack{x\in\Z^d:\\ R\leq |x|\leq 2R}}\sum_{\substack{e\in\mathbb{Z}^d\\|e|=1}} |G(x)-G(x+e)|^2.
\end{aligned}
\eeq

To bound the annular averages of discrete derivatives of $G$, we employ \cite[Lemma~1.4]{MO2}, which gives the existence of a constant $\tilde C_{d,\lam}>0$ such that for every $a\in\curly{A}_{\lam}$ and every radius $R\geq 1$,
$$
R^{-d}\sum_{\substack{x\in\Z^d:\\ R\leq |x|\leq 2R}}\sum_{\substack{e\in\mathbb{Z}^d\\|e|=1}} |G(x)-G(x+e)|^2\leq \tilde C_{d,\lam}R^{2(1-d)}.
$$
Using this estimate on \eqref{eq:justaronson} proves the desired auxiliary bound \eqref{eq:mainspatialub}. 

To conclude the upper bound in \eqref{eq:mainspatial}, we fix $\ell_d$ as determined by the lower bound and cover the annulus $\{x:\, R\leq|x|\leq  \ell_d R\}$ with finitely many dyadic annuli to obtain
$$
 R^{-d}\sum_{\substack{x\in\curly S^{j,\al}_{R,\ell_d}}} w_G(x)
 \leq R^{-d}\sum_{\substack{x\in\Z^d:\\ R\leq |x|\leq \ell_d R}} w_G(x)
 \leq c_{d,\lam} R^{-2},
$$
since $\curly S^{j,\al}_{R,\ell_d} \subseteq \{x:\, R\leq|x|\leq \ell_d R\}$. This completes the proof of Theorem \ref{thm:mainspatial}.
\qed

\subsection{Proof of Proposition \ref{prop:continuum}}
\label{sect:continuumpf}
The proof follows along the same lines as Theorem \ref{thm:mainspatial}. Here we only summarize the necessary modifications.

First, in the continuum, the Hardy weight given in \eqref{eq:WG} is already of the correct form, so an analog of Proposition \ref{prop:wGbounds} is not needed. 

For the upper bound, we then use the Nash-Aronson bound \eqref{eq:aronson} in the original continuum version and note that Lemma 1.4 of \cite{MO2} is also proved for the continuum in Section 2 of \cite{MO2}; see also Lemma 2.1 in \cite{LSW63}. For the lower bound, the comparison between sectors and cuboids is analogous and for a continuum analog of Lemma \ref{lm:onedirection}, it suffices to replace the telescopic sum in \eqref{eq:telescopic} by an application of the fundamental theorem of calculus. The details are left to the reader.
\qed

\section{Inverse-square upper bounds for ergodic coefficients}
\label{sect:mainrandompf}
In this section, we prove Theorem \ref{thm:mainrandom} for general ergodic coefficients subject to Assumption \ref{ass:P}. 

\be{proof}[Proof of Theorem \ref{thm:mainrandom}]
We first prove the annealed estimate \eqref{eq:wGannealed}.  Let $p\geq 1$. The case $ x=0 $ is clear by   Proposition~\ref{prop:wGbounds}. For $x\neq 0$, by  Proposition~\ref{prop:wGbounds}  and the Nash-Aronson bound \eqref{eq:aronsondiscrete}, we have
\begin{align*}
 \qmexp{w_G(x)^p}
&\leq
\left\langle\frac{1}{G(x)^{2p}}
\left(\sum_{\substack{e\in\mathbb{Z}^d, |e|=1}} a([x,x+e])|G(x)-G(x+e)|^2\right)^{p}
	\right\rangle\\
		&
	\leq
\left\langle
C_{d,\lam}^{2p} \l(1+|x|\r)^{2p(d-2)}
\left(\sum_{\substack{e\in\mathbb{Z}^d,|e|=1}} a([x,x+e])|G(x)-G(x+e)|^2\right)^{p}
	\right\rangle.
\end{align*}
To estimate the discrete derivative of $G$, we use $|a|\leq 1$ and \cite[Theorem 1]{MO1} which applies under Assumption \ref{ass:P} and gives
\beq\label{eq:MO1}
\qmexp{|G(x)-G(x+ e)|^{2p}}^{\frac{1}{2p}}\leq C'_{d,\lam,\rho,p} (1+|x|)^{1-d},
\eeq
for every  $ e\in\mathbb{Z}^{d} $, $|e|=1$. Using this and the elementary inequality $\l(\sum_{j=1}^{2d} b_j\r)^{p}\leq (2d)^{p-1} \sum_{j=1}^{2d} b_j^p$, we find
$$
\begin{aligned}
 \qmexp{w_G(x)^p}
\leq& 
C_{d,\lam}^{2p} (2d)^{p-1} \l(1+|x|\r)^{2p(d-2)}\sum_{\substack{e\in\mathbb{Z}^d,|e|=1}} \qmexp{|G(x)-G(x+e)|^{2p}}\\
\leq & C_{d,\lam}^{2p}(2d)^{p}d(C'_{d,\lam,\rho,p})^{2p} (1+|x|)^{-2p}.
\end{aligned}
$$
This proves \eqref{eq:wGannealed} when $x\neq 0$. 

It remains to prove the almost-sure estimate \eqref{eq:wGprob1}. Let $\eps>0$ and set $p_\eps=\frac{2d}{\eps}$. By Markov's inequality and \eqref{eq:wGannealed}, we have
$$
\sum_{x\in \Z^d}\mathbb P\l(w_G(x)\geq (1+|x|)^{-2+\eps}\r)
\leq 
C_{d,\eps,\lam}\sum_{x\in \Z^d}(1+|x|)^{-2\eps p_\eps}<\infty.
$$
Now, the Borel-Cantelli lemma implies that $w_G(x)\geq (1+|x|)^{-2+\eps}$ can occur for at most finitely many $x\in \Z^d$.  This proves Theorem \ref{thm:mainrandom}.
\e{proof}

\section{Proof of Theorem \ref{thm:iid} for i.i.d.\ coefficients}
\label{sect:iidpf} 
We come to the proof of Theorem~\ref{thm:iid} for which we assume the coefficients are generated by \eqref{eq:aiid}. In the following, we suppress $\omega$ from the notation. 

Without loss of generality, we let $x\neq 0$. We begin by rewriting \eqref{eq:wG} as
\begin{align*}
w_G(x)&=\frac{1}{G(x)}
	\sum_{\substack{e\in \Z^d, |e|=1}}a([x,x+e]) (G(x)^{\frac{1}{2}}-G(x+e)^{\frac{1}{2}})^{2}\\
	&=\frac{1}{G(x)} \sum_{\substack{e\in \Z^d, |e|=1}} 
	F(\zeta_{[x,x+e]})\l(\frac{G(x)-G(x+e)}{\sqrt{G(x)}+\sqrt{G(x+e)}}\r)^{2}.
\end{align*}  
\subsection{Step 1: Approximating $G$ by the homogenized Green's function}
 In this step, we replace the Green's function $G$ (but not its discrete derivative) to leading asymptotic order by the homogenized Green's function $G_h$ whose definition we recall now. We fix a function $f:\R^d\to \R^d$ and for $\eps>0$ we define the rescaled discrete gradient by,
\begin{align*}
	\nabla_\eps :C(\eps\Z^{d})\to C(\eps \Z^{d})^{d}, \quad 
	\nabla_\eps f(x )=
	\frac{1}{\eps}\left(\begin{matrix}
		f(x+\eps e_{1})-f(x)\\
		\vdots\\
		f(x+\eps e_{d})-f(x)
	\end{matrix}\right).
\end{align*}
We write $\nabla_\eps^*$ for its $\ell^2(\eps\Z^d)$-adjoint. For uniformly elliptic and sufficiently ergodic coefficients, so in particular for $a(e)$ given by \eqref{eq:aiid}, the solution $u_\eps(x)=\eps^2 u(x/\eps)$ of the rescaled equation 
$$
\nabla_\eps^* a\l(\frac{\cdot}{\eps}\r) \nabla_\eps u_\eps(x)=f(x), \qquad x\in \eps\Z^d, 
$$
converges as $\eps\to 0$ to a function $u_{h}:\R^d\to\R$ solving the \textit{homogenized} equation
$$
\mathrm{div}(a_{h} \nabla u_{h})(x)=f(x), \qquad x\in \R^d, 
$$
where $a_{h}\in \R^{d\times d}$ is a symmetric \textit{constant} coefficient matrix and $\lam\leq a_{h}\leq \lam^{-1}$ is uniformly elliptic. (This fact is known as qualitative (i.e., non-quantitative) homogenization and goes back to Papanicolau-Varadhan \cite{PV0}.) We write $G_h:\R^d\to\R_+$ for the Green's function of $\mathrm{div}(a_{h} \nabla)$. It is called the homogenized Green's function and by a change of variables in Fourier space it is of the form
\beq\label{eq:Ghformula}
G_h(x)=C_{univ} (\det(a_h))^{d/2} |a_{h}^{-1/2} x|^{2-d}
\eeq
with $C_{univ}>0$ an explicit universal constant. 
 
We introduce the auxiliary function
$$
\tilde w(x)=\frac{1}{4G_h(x)^2} \sum_{\substack{e\in \Z^d, |e|=1}} 
	F(\zeta_{[x,x+e]}) (G(x)-G(x+e))^2.
$$

\be{lemma}\label{lm:tildew}
Let $p\geq 1 $. For every $\de\in (0,1)$, there exist $C_{d,\de,\lam,p}>0$ so that for all $x\neq 0$,
\beq
\qmexp{|w_G(x)-\tilde w(x)|^p}^{\tfrac{1}{p}}\leq C_{d,\de,\lam,p} |x|^{\tfrac{\de-1}{2}} |x|^{-2}.
\eeq
\e{lemma}

\be{proof}[Proof of Lemma \ref{lm:tildew}]
For convenience, we work with the strongest available version of homogenization of the Green's function which was developed in \cite[Chapters 8 \& 9]{AKM} based on scalar de Giorgi-Nash $L^\infty$ estimates.  (An alternative would be to use the earlier concentration bound \cite[Corollary 1]{MO1}.) While \cite{AKM} works in the continuum setting throughout, the techniques extend to the discrete setting, even to supercritical percolation clusters \cite{DG}. In our case, it suffices to note that the following result holds by taking $\mathfrak{p}=1$ in \cite[Theorem 2]{DG}.

For every $\de>0$, there exist $s=s_{d,\de,\lam},C=C_{d,\de,\lam}>0$ and a non-negative random variable $\mathcal M_\de$ satisfying
\beq\label{eq:exptails}
\mathbb P(\mathcal M_\de\geq R)\leq C \exp\l(-C ^{-1}R^s\r),\qquad R\geq 0,
\eeq
so that for every $x\in\Z$ satisfying $|x|\geq \mathcal M_\de$, we have the homogenization of the Green's function
\beq\label{eq:GFhom}
|G(x)-G_h(x)|\leq C |x|^{1-d+\de}.
\eeq
(Compare the Nash-Aronson bounds \eqref{eq:aronsondiscrete}.) We note that \eqref{eq:Ghformula} and uniform ellipticity of $a_h$ imply that all derivatives of $G_h$ have the correct scaling behavior, i.e., there exists $C>1$ so that
\beq\label{eq:Ghderivest}
C^{-1} |x|^{1-d}\leq |\nabla G_h(x)|\leq C |x|^{1-d}.
\eeq
We now fix $\de>0$. Suppose that $|x|\geq \mathcal M_\de+1$. Then by combining $|F|\leq \lam^{-1}$, \eqref{eq:GFhom}, \eqref{eq:Ghderivest}, and \eqref{eq:aronsondiscrete} with elementary estimates such as
\begin{align*}
	\sqrt{G(x)G(x+e)}-&G_h(x)\\
\leq &
\sqrt{|G(x)-G_h(x)| G_h(x+e)}+\sqrt{G_h(x)}\sqrt{|G(x+e)-G_h(x+e)|}\\
&+\sqrt{G_h(x)}\sqrt{|G_h(x+e)-G_h(x)|}\\
\leq& C |x|^{\tfrac{\de-1}{2}} |x|^{2-d},
\end{align*}
we obtain
\beq\label{eq:tildewcompare}
|w_G(x)-\tilde w(x)|\leq C \frac{|x|^{\tfrac{\de-1}{2}}}{G_h(x)^2} \sum_{\substack{e\in \Z^d, |e|=1}} (G(x)-G(x+e))^2.
\eeq
We fix $p\geq 1$ and write
\begin{multline*}
{\qmexp{|w_G(x)-\tilde w(x)|^p}^{\tfrac{1}{p}}}\\
\leq C_{d,\lam} \max_{x\in \mathbb{Z}^{d},\omega \in \Omega}\{|w_G(x,\omega)|,|\tilde w(x,\omega)|\} \mathbb P(\mathcal M_\de\geq |x|) \\
+\qmexp{|w_G(x)-\tilde w(x)|^p \mathbbm 1_{\{|x|\geq \mathcal M_\de+1\}}}^{\tfrac{1}{p}}.
\end{multline*}
We estimate the first term by \eqref{eq:aronsondiscrete} and  \eqref{eq:exptails} and the second term by \eqref{eq:tildewcompare} and \eqref{eq:MO1}. We obtain
\begin{align*}
\lefteqn{\qmexp{|w_G(x)-\tilde w(x)|^p}^{\tfrac{1}{p}}}\\
&\leq  C \exp\l(-C ^{-1}|x|^{-s}\r) 
+C 2^{1-\tfrac{1}{p}} \frac{|x|^{\tfrac{\de-1}{2}}}{G_h(x)^2} \l(\sum_{\substack{e\in \Z^d, |e|=1}} (G(x)-G(x+e))^{2p}\r)^{\tfrac{1}{p}}\\
&\leq C_{d,\de,\lam,p} |x|^{\tfrac{\de-1}{2}} |x|^{-2}.
\end{align*}
This proves Lemma \ref{lm:tildew}.
\e{proof}

\subsection{Step 2: Approximating $\nabla G$ through correctors}
In Step 2, we express the Green's function gradient that appears in $\tilde w(x)$ through the correctors from stochastic homogenization by using a two-scale expansion. The correctors are defined as follows. For $\xi\in\R^d$, let $\phi_\xi:\Z^d\to\R$ denote the unique stationary solution of
$$
\nabla^* a (\xi+\nabla\phi_\xi)=0
$$
satisfying $\qmexp{\phi_\xi(x)}=0$ for all $x\in\Z^d$. (See \cite{AKM} and \cite[Lemma~2.1]{GO},  \cite[Theorem~3]{Ku} in the discrete setting for proofs that correctors exist.) Denote $\phi_k\equiv \phi_{e_k}$ for $1\leq k\leq d$. We introduce the auxiliary functions
\beq\label{eq:hatwdefn}
\begin{aligned}
\hat w(x)=&\frac{1}{4G_h(x)^2} \sum_{j=1}^d 
\l[ F(\zeta_{[x,x+e_j]}) h_j(x)^2 +F(\zeta_{[x-e_j,x]}) h_j(x-e_j)^2
\r],\\
h_j(x)=&\nabla_j G_h(x)+\sum_{k=1}^d \nabla_k G_h(x) \nabla_j\phi_k(x)
\end{aligned}.
\eeq
and, for $ f:\Z^{d}\to \mathbb{R} $,
\begin{align*}
	\nabla _{j}f(x)=f(x+e_{j})-f(x),\qquad j=1,\ldots,d.
\end{align*}

\be{lemma}\label{lm:hatw}
Let $p\geq 1 $. There exists $C_{d,\lam,p}>0$ so that
$$
\qmexp{|\tilde w(x)-\hat w(x)|^p}^{1/p}\leq  C_{d,\lam,p} \frac{\log |x|}{|x|^{2+1/p}} ,\qquad x\in \Z^d.
$$
\e{lemma}

\be{proof}[Proof of Lemma \ref{lm:hatw}]
By Proposition 4.2 in \cite{GM}, we have the two-scale expansion
\begin{align*}
	\qmexp{\left|\nabla_j G(x)-h_j(x)\right|^2}^{1/2}
\leq C_{d,\lam} \frac{\log |x|}{|x|^d}, \qquad j\in \{1,\ldots,d\}.
\end{align*}
Without loss of generality, we assume $|x|> 2$. Let $p\geq 1$. By  \eqref{eq:aronsondiscrete}, the triangle inequality for $\qmexp{(\cdot)^p}^{1/p}$ and the Cauchy-Schwarz inequality for $\qmexp{(\cdot)^2}^{1/2}$, we obtain
\begin{align*}
\lefteqn{\qmexp{|\tilde w(x)-\hat w(x)|^p}^{1/p}}\\
\leq &\frac{\lam^{-1}}{4G_h(x)^{2}} \sum_{j=1}^d \sum_{t\in \{0,1\}} \qmexp{|(\nabla_j G(x-te_j))^2-h_j(x-te_j)^2|^p}^{1/p} \\
\leq &\frac{C_{d,\lam,p}}{|x|^{4-2d}} \sum_{j=1}^d \sum_{t\in \{0,1\}} \qmexp{|\nabla_j G(x-te_j)-h_j(x-te_j)|^2}^{1/(2p)}\\
&\qquad\qquad\qquad\qquad \qmexp{|\nabla_j G(x-te_j)|^{4p-2}+|h_j(x-te_j)|^{4p-2}}^{1/(2p)} \\
\leq& C_{d,\lam,p}  \frac{(\log|x|)^{1/(2p)}}{|x|^{4-2d+d/p}}\\
&\qquad\max_{1\leq j\leq d,\ t\in \{0,1\}} \l(\qmexp{|\nabla_j G(x-te_j)|^{4p-2}}^{1/(2p)}+\qmexp{|h_j(x-te_j)|^{4p-2}}^{1/(2p)}\r)\\
\leq& C_{d,\lam,p} \frac{\log |x|}{|x|^{2+1/p}}.
\end{align*}
The last estimate uses \eqref{eq:MO1}, \eqref{eq:Ghderivest}, and the fact that the correctors $\phi_k$ are sublinear in the sense that $\qmexp{|\nabla \phi_k(x)|^q}^{1/q}\leq C_{d,\lam,q}$ for all $q\geq 1$, cf.\ \cite[Proposition~2.1]{GO}. This proves Lemma \ref{lm:hatw}.
\e{proof}

\subsection{Step 3: Moment lower bounds for large $|x|$} 
Let $p\geq 1$ and $\de=\tfrac{1}{2}$. Combining Lemmas \ref{lm:tildew} and \ref{lm:hatw} with the triangle inequality for $\qmexp{(\cdot)^p}^{1/p}$, we obtain
\beq\label{eq:preciseerror}
\begin{aligned}
\qmexp{w_G(x)^p}^{1/p}\geq& \qmexp{\hat w(x)^p}^{1/p}- \frac{C_{d,\lam,p}}{|x|^2} \max\l\{\frac{1}{|x|^{1/4}},\frac{\log|x|}{|x|^{1/p}}\r\}\\
\geq& \qmexp{\hat w(x)^p}^{1/p}-\frac{C_{d,\lam,p}}{|x|^{2+\min\{\tfrac{1}{4},\tfrac{1}{2p}\}}}.
\end{aligned}
\eeq
To estimate the main term $\qmexp{\hat w(x)^p}^{1/p}$, we use $F(\zeta_e)\geq \lam$, the deterministic bound \eqref{eq:Ghderivest}, and Jensen's inequality applied to the convex function $z\mapsto z^{2p}$. This gives
$$
\begin{aligned}
\qmexp{\hat w(x)^p}
\geq& C_{d,\lam,p} (1+|x|)^{2p(d-2)}\qmexp{\l(\sum_{j=1}^d h_j(x)^2 +h_j(x-e_j)^2\r)^p}\\
\geq& C_{d,\lam,p} (1+|x|)^{2p(d-2)}\sum_{j=1}^d \qmexp{h_j(x)^{2p}}\\
\geq& C_{d,\lam,p} (1+|x|)^{2p(d-2)}  \sum_{j=1}^d \qmexp{h_j(x)}^{2p}.
\end{aligned}
$$
Recalling \eqref{eq:hatwdefn} and $\qmexp{\phi_k(x)}=0$ and using \eqref{eq:Ghderivest} once again, we find
$$
\begin{aligned}
\qmexp{\hat w(x)^p}
\geq& C_{d,\lam,p} (1+|x|)^{-2p}.
\end{aligned}
$$
Combining this with \eqref{eq:preciseerror}, we have shown that there exists a radius $R_{d,\lam,p}>0$ so that
\beq\label{eq:largex}
\begin{aligned}
\qmexp{w_G(x)^p}^{1/p}
\geq C_{d,\lam,p} (1+|x|)^{-2}\qquad  |x|\geq R_{d,\lam,p}.
\end{aligned}
\eeq
This proves the assertion \eqref{eq:mainrandompert1} of Theorem \ref{thm:iid}. For small distances and $d\geq 5$, we can prove a uniform lower bound by different techniques, namely perturbation theory for the Green's function, as we described next.

\subsection{Step 4: Moment lower bounds for small $|x|$} 
Assume that $d\geq 5$.
This proves \eqref{eq:mainrandompert1}. Given \eqref{eq:mainrandompert1}, extending the estimate $\qmexp{w_G(x)^p}^{1/p}
\geq C_{d,\lam,p} (1+|x|)^{-2}$ to all $ x $ requires a pointwise positive lower bound on $\qmexp{w_G(x)^p}$ over the ball $B_{R_{d,\lam,p}}$ where
$$
B_{R}=\setof{x\in\Z^d}{|x|<R}.
$$
 
Recall that $G_0$ denotes the free Green's function. The key idea for proving this is that with positive probability $G$ is uniformly close to some $\beta^{-1} G_0$ on the entire ball $B_{R_0}$.

\be{lemma}[Green's function comparison]
\label{lm:GFcomparison}
Let $d\geq 5$ and $\eps,R_0>0$. There exists $\beta\in [\lam,1]$ and $p_{\lam,\eps,R_0}>0$ so that
\begin{align*}
	\mathbb P\l(\sup_{x_1\in B_{2R_0}}|G(x_1)-\beta^{-1} G_0(x_1)|\leq  \eps\r)\geq p_{\lam,\eps,R_0}>0.
\end{align*}
\e{lemma}

We postpone the proof of this lemma for now.

\be{proof}[Proof of Theorem \ref{thm:iid}]
The estimate \eqref{eq:mainrandompert1} was proved in \eqref{eq:largex} above. Let $d\geq 5$ and fix $x_0\in B_{R_0}$. Combining Proposition \ref{prop:wGbounds} and the Nash-Aronson estimate \eqref{eq:aronsondiscrete}, there exists $C_{d,\lam,p} >0$ so that
\beq\label{eq:wGweak}
w_G(x_0)\geq C_{d,\lam} \sum_{\substack{y\in\Z^d, y\sim x_0}}|G(x_0)-G(y)|.
\eeq
We now investigate whether $|G(x_0)-G(y)|$ can vanish. We apply the Green's function comparison, Lemma \ref{lm:GFcomparison} with $R_0=R_{d,\lam,p}$ and $\eps>0$ to be determined later. It says that there exists $\beta\in [\lam,1]$ so that with probability $p_{\lam,\eps,d,p}>0$, we have
\beq\label{eq:GFreplacement}
\sum_{\substack{y\in\Z^d, y\sim x_0}}|G(x_0)-G(y)|
\geq \beta^{-1}\sum_{\substack{y\in\Z^d, y\sim x_0}}|G_0(x_0)- G_0(y)|-4d\eps.
\eeq
The deterministic Lemma \ref{lm:G0} (``The free Green's function is never locally constant'') implies that there exists $C^{\mathrm{free}}_{d,R_0}>0$ so that
$$
\beta^{-1}\sum_{\substack{y\in\Z^d, y\sim x_0}}|G_0(x_0)- G_0(y)|
>\beta^{-1}C^{\mathrm{free}}_{d,R_0}.
$$
Now we choose $\eps=\eps_{d,R_0}=\eps_{\lam,d,p}$ small enough such that $\beta^{-1}C^{\mathrm{free}}_{d,R_0}-4d\eps\geq {C^{\mathrm{free}}_{d,R_0}}/({2\beta})$.

Let $\Omega_0$ denote the event that $\sum_{\substack{ y\sim x_0}}|G(x)-G(y)|>{C^{\mathrm{free}}_{d,R_0}}/({2\beta})$. We have shown that
\begin{align*}
\mathbb P\l(\Om_0\r) >p_{\lam,d,p}>0.
\end{align*}
 By \eqref{eq:wGweak} and positivity, we find that 
 $$
 \qmexp{w_G(x_0)^p}^{1/p}\geq C_{d,\lam} 
 \qmexp{\mathbbm 1_{\Om_0} \sum_{\substack{y\in\Z^d, y\sim x_0}}|G(x_0)-G(y)|}> C_{d,\lam}  \frac{C^{\mathrm{free}}_{d,R_0}}{2\beta}p_{\lam,d,p}>0.
 $$
 Since $x_0\in B_{R_0}$ was arbitrary and $R_0=R_{d,\lam,p}$, we have proven Theorem~\ref{thm:iid}.
\e{proof}

It remains to prove  the Green's function comparison, Lemma \ref{lm:GFcomparison}.
\be{proof}[Proof of Lemma \ref{lm:GFcomparison}]
Let $\eps>0$. We choose $\beta\in [\lam,1]$ to be a point in the support of the random variable $F(\zeta_{e_1})$ so that $\mathbb P(F(\zeta_{e_1})\in (\beta-\eps,\beta+\eps))$ is maximal. A covering argument shows that
$$
\mathbb{P}\left(F(\zeta_{e_1})\in (\beta-\eps,\beta+\eps)\right)\geq p_{\lam,\eps}>0.
$$
with $p_{\lam,\eps}$ depending only on $\lam,\eps$ and not on the law of $F(\zeta_{e_1})$.
We define the larger radius
$$ R_\eps=\max\{4R_0,\eps^{-2}\} .$$ 
We say $e \in B_{R_\eps}$ if $e=(x,y)$ and $x\in B_{R_\eps}$ or $y\in B_{R_\eps}$.
Since the $\{F(\zeta_{e})\}_{e\in \mathbb E^d}$ are i.i.d.\ random variables, we can strengthen the above to the statement that there exists $ p_{R_0,\lam,\eps}>0$ so that
\beq\label{eq:betasupport}
\mathbb{P}\l(F(\zeta_{e})\in (\beta-\eps,\beta+\eps) \mbox{ for all } e \in B_{R_\eps} \r)\geq  p_{\lam,\eps,R_0}>0.
\eeq
From now on, we fix a realization of $\{F(\zeta_{e})\}_{e\in \mathbb E^d}$ which we assume satisfies $F(\zeta_{e})\in (\beta-\eps,\beta+\eps)$ for all $e\in B_{R_\eps}$.

We now replace the coefficients $a(e)=F(\zeta_e)$ in the Green's function one-by-one by the constant $\beta$ via the resolvent identity. To express this compactly, we introduce convenient notation from \cite{MO1}. Let $F\subset \mathbb E^d$ and let $F'=F\cup \{e'\}$ for some fixed edge $e'\not\in F$. For $\gam\in\{F,F'\}$ let $a^\gam=\{a^{\gam}(e)\}_{e\in \mathbb E^d}$ be two families of coefficients in $\mathcal A_{\lam}$ satisfying
$$
a^F(e')\neq a^{F'}(e')  \quad \textnormal{ and } \quad a^F(e)= a^{F'}(e),\qquad  e\neq e'.
$$
We write $G^\gam$, $\gam\in\{F,F'\}$, for the corresponding Green's functions. The resolvent identity gives for $ x\in \Z^d $
\beq\label{eq:resolventidentity}
G^F(x,0)-G^{F'}(x,0)=(a^F(e')- a^{F'}(e)) \nabla G^{F}(x,e') \nabla G^{F'}(e',0),\eeq
cf.\ equation~(53) in \cite{MO1}.
Let $\{e_n\}_{n\geq 0}=\mathbb E^d$ be an enumeration of the edge set. We define
$$
F_N:=\bigcup_{1\leq n\leq N} \{e_n\},\qquad  N\geq 0
$$
with the convention that $F_0=\emptyset$ and $F_\infty=\mathbb E^d$. We note that every edge $e$ has a unique representation as $e=[x,x+e_j]$ with $x\in\Z^d$ and $j\in \{1,\ldots,d\}$. We use this to define the modified coefficients
\beq\label{eq:aedefn}
a^{F_N}(e)=
\be{cases}
a(e),\qquad &\textnormal{if } e=[x,x+e_j]\in F_N,\\
 \beta,\qquad &\textnormal{if } e\not\in F_N.
\e{cases}
\eeq

We can then use telescoping and the resolvent identity \eqref{eq:resolventidentity} to write, for every $x_1\in B_{2R_0}$,
\beq\label{eq:GminusGbeta}
\begin{aligned}
	&G(x)-\beta^{-1} G_0(x_1)=G^{F_0}(x_1)-G^{F_\infty}(x_1)\\
	&=\sum_{N\geq 0} G^{F_N}(x_1)-G^{F_{N+1}}(x_1)\\
	&=\sum_{N\geq 0} (a^{F_N}(e_{N+1})- a^{F_{N+1}}(e_{N+1})) \nabla G^{F_N}(x_1,e_{N+1}) \nabla G^{F_{N+1}}(e_{N+1},0).
\end{aligned}
\eeq
We decompose the last sum in $N$ as follows. Let $\mathcal I\subset \mathbb N_0$ be the finite index set of non-negative integers $N\geq 0$ so that $e_{N+1}=[x,x+e_j]$ with $x\in B_{R_\eps}$. On this set, we may apply \eqref{eq:betasupport} and use \eqref{eq:aedefn} to obtain
$$
|a^{F_N}(e_{N+1})- a^{F_{N+1}}(e_{N+1})|\leq \eps
$$
On $B_{R_\eps}$, we also use the simple a priori estimate
\beq\label{eq:GFapriori}
|\nabla G^{F_N}(x_1,e_{N+1}) \nabla G^{F_{N+1}}(e_{N+1},0)|
\leq 4 \l(\sup_{N\geq 0} \sup_{x,y\in \Z^d} G^{F_N}(x,y)\r)^2
\leq C_{d,\lam}.
\eeq
which holds by \eqref{eq:aronsondiscrete}. Employing these estimates for $n\in\mathcal I$ in \eqref{eq:GminusGbeta} gives

\begin{multline}\label{eq:GimplementI}
	|G(x_1)-\beta^{-1} G_0(x_1)|\\
	\leq\eps | \mathcal I| C_{d,\lam}+\sum_{N\in \mathbb N_0\setminus \mathcal{I}}
	|\nabla G^{F_N}(x_1,e_{N+1})| |\nabla G^{F_{N+1}}(e_{N+1},0)|,
\end{multline}
where $| \mathcal I|<\infty$ denotes the cardinality of $\mathcal I$. 

To estimate the sum over $\mathbb N_0\setminus \mathcal{I}$, we first reparametrize it using the fact that there is a bijective map $N:B_{R_\eps}^c \times \{1,\ldots,d\}\to \mathbb N_0\setminus \mathcal I$ so that $e_{N(x,j)}=[x,x+e_j]$.
\begin{multline*}
	\sum_{N\in \mathbb N_0\setminus \mathcal{I}} |\nabla G^{F_N}(x_1,e_{N+1})| |\nabla G^{F_{N+1}}(e_{N+1},0)|\\
	=\sum_{x\in B_{R_\eps}^c} \sum_{j=1}^d |\nabla G^{F_{N(x,j)}}(x_1,[x,x+e_j])| |\nabla G^{F_{N(x,j)+1}}([x,x+e_j],0)|.
\end{multline*}
Next, we refine the a priori estimate \eqref{eq:GFapriori} at large distances via \eqref{eq:aronsondiscrete}  to
$$
\begin{aligned}
	|\nabla G^{F_{N(x,j)}}(x_1,[x,x+e_j])| |\nabla G^{F_{N(x,j)+1}}([x,x+e_j],0)|\leq C_{d,\lam} |x_1-x|^{2-d} |x|^{2-d}.
\end{aligned}
$$
We observe that $|x-x_1|\geq |x|/2$ for all $ x\in B_{R_{\eps}^{c}} $ because $x_1\in B_{2R_0}$ and $R_\eps\geq 4R_0$. Hence $|x_1-x|^{2-d}\leq C |x|^{2-d}$ 
and  since $ \sum_{x\in \Z^{d}} |x|^{-d-1/2}\leq C_{d}<\infty$ we have shown that
$$
\sum_{x\in B_{R_\eps}^c} \sum_{j=1}^d |\nabla G^{F_{N(x,j)}}(x_1,[x,x+e_j])| |\nabla G^{F_{N(x,j)+1}}([x,x+e_j],0)|
\leq C_{d,\lam} R_\eps^{4-d-{1}/{2}}.
$$
Whenever $ d\ge5 $, we have $  R_\eps^{4-d-{1}/{2}}\leq \eps $ since $R_\eps\geq \eps^{-2}$. 
In view of  \eqref{eq:GimplementI}, this implies
$$
|G(x_1)-\beta^{-1} G_0(x_1)|\leq C_{d,\lam,R_0} \eps.
$$
Note that all estimates are uniform in $x_1\in B_{2R_0}$ once we assume that \eqref{eq:betasupport} occurs. This proves Lemma \ref{lm:GFcomparison}.
\e{proof}


\section{Application to Rellich inequalities}\label{sect:rellich}
In this short final section, we explain the implications of the results from Section~\ref{sect:main2} for Rellich inequalities on graphs. For background, we recall that the original ``Rellich inequality'' was presented by F.~Rellich at the ICM 1954 in Amsterdam \cite{R}. Very recently, \cite{KePiPo4} provided a general mechanism for generating Rellich inequalities from strictly positive Hardy weights on graphs (see \cite{Ro} for the case of strongly local Dirichlet forms). Adapting \cite[Corollary 4.1]{KePiPo4} to the case of elliptic operators on $\Z^d$ and the subharmonic function $u=G$, we obtain the following result.

\be{thm}[Rellich inequality on $\Z^d$, \cite{KePiPo4}]\label{thm:rellich} 
Let $ a\in \mathcal{A}_{\lambda} $ for $ \lambda\in(0,1] $, $\mathcal S=\supp w_G\subseteq \Z^{d}$  and  $\al\in (0,1)$. Then we have the Rellich inequality
\beq\label{eq:rellich}
\|\mathbf{1}_\varphi \Delta \varphi \|_{\frac{G^\al}{w_G}} \geq (1-\gam) \|\varphi\|_{G^\al w_G},\qquad \ph\in C_{c}(\mathcal S),
\eeq
where $\mathbf{1}_{\ph} $ is the characteristic function of the support of $ \ph $ and the constant is
$$
\gam=\l(\frac{1-(\lambda^{2}/2d)^{\al/2}}{1-(\lambda^{2}/2d)^{1/2}}\r)^2.
$$
\e{thm}
\begin{proof}
In view of \cite[Remark~3.4]{KePiPo4} the assumption of standard weights can replaced by the ellipiticity condition \eqref{E} which is satisfied in our setting with $ E= \lambda^{2}/2d $. Precisely, in  \cite[Theorem~3.3]{KePiPo4} from which   \cite[Corollary 4.1]{KePiPo4} is deduced one can replaced the constant $ D $ with  $ 1/E $. 
\end{proof}

For the free Laplacian on $\Z^d$, $d\geq 5$, and $\al=\tfrac{2}{d-2}$ the standard Green's function asymptotic shows that $\frac{G_0^\al}{w_{G_0}} \sim C$ and $G_0^\al w_{G_0}\sim C'  |x|^{-4}$ \cite[Example 4.5]{KePiPo4}. This matches the $|x|^{-4}$-scaling of the original Rellich weight on $\R^d$, \cite{R}.

We consider this question in the i.i.d.\ random setting of Section \ref{ssect:iid}. They allow us to extend the $|x|^{-4}$ from the free Laplacian to these random elliptic coefficients, again in a probabilistic sense.


\be{proposition}[$|x|^{-4}$-scaling for Rellich weights]
For $d\geq 5$, let $\al=\tfrac{2}{d-2}$. Then, there exists $C_{d,\lam}>0$ such that
$$
\begin{aligned}
\left\langle G(x)^\al w_G(x)  \right\rangle \geq C_{d,\lam} (1+|x|)^{-4}
\end{aligned}
$$
and there exists   $ c_{d,\lam}>0 $ such that
$$
\mathbb P\l(\frac{G(x)^\al}{w_G(x)}<C_{d,\lam}\r)\geq c_{d,\lam},\qquad  x\in \Z^d.
$$
\e{proposition}

\be{proof}
The first bound follows from \eqref{eq:aronsondiscrete} and Theorem \ref{thm:iid}. 
The second bound follows  from \eqref{eq:aronsondiscrete} and Corollary \ref{cor:smm}.
\e{proof}

Note that the direction of the bounds are the right ones to be applicable in \eqref{eq:rellich}. Namely, on the left-hand side of \eqref{eq:rellich}, the weight will be constant, while on the right-hand side it behaves as $|x|^{-4}$ for large $x$, exactly as for the discrete Laplacian and as in Rellich's original inequality \cite{R}.

\section*{Acknowledgments} The authors are grateful to Scott Armstrong and Yehuda Pinchover for very valuable discussions. They would also like to thank Peter Bella, Mitia Duerinckx, Arianna Giunti, and Felix Otto for helpful comments on a draft version of the paper. MK acknowledges the financial support of the German Science Foundation. The authors would like to thank the organizers of the program ``Spectral Methods in Mathematical Physics'' held in 2019 at Institut Mittag-Leffler where this project was initiated. 
\begin{appendix}

%
%
%
%

\section{The free Green's function is never locally constant}\label{app:G0}
The following lemma is used in the main text to derive lower bounds on Hardy weights near the origin. It is elementary, but does not appear to be completely standard. We learned of the argument through \cite{mathoverflow} and include it here for the convenience of the reader.

We recall that $G_0$ is the Green's function of the free Laplacian on $\Z^d$.

\be{lemma}\label{lm:G0}
Let $d\geq 1$. Then
\begin{align*}
	\sum_{\substack{y\in\Z^d:\\ y\sim x}}|G_0(x)-G_0(y)| >0,\qquad  x\in\Z^d.
\end{align*}
\e{lemma}

\be{proof}
We write $\mathbb{P}^0$ for the probability measure of the symmetric simple  random walk $ S $ started at the origin. A random-walk representation of the Green's function is given by
$$
G_0(x)=C_d\sum_{n\geq 0} \mathbb{P}^0(S_{n}=x).
$$
Let $ x\in \Z^{d} $ and without loss of generality, we suppose that the first coordinate $x_1\geq 1$. Moreover, we assume that $\sum_{i=1}^d x_i$ is odd; the even case can be argued analogously. Observe that then
$$
\mathbb{P}^0(S_{n}=x+e_1)=\mathbb{P}^0(S_{n}=x-e_1)=0,\qquad  n\geq 0 \textnormal{ odd},
$$
due to parity considerations.

Now we claim that if $\mathbb{P}^0(S_{n}=x+e_1)>0$, then
\beq\label{eq:G0claim}
\mathbb{P}^0(S_{n}=x+e_1)<\mathbb{P}^0(S_{n}=x-e_1),\qquad  n\geq 0 \textnormal{ even}.
\eeq
This implies the assertion of the lemma after summation in $n$.

It remains to prove \eqref{eq:G0claim}. Given an index set $\mathcal I_n\subset \{1,\ldots,n\}$, let $A(\mathcal I_n)$ be the random event that the steps $\mathcal I_n$ occurr in the $\pm e_1$ direction and the steps in $\{1,\ldots,n\}\setminus \mathcal I_n$ do not. By conditioning we have
$$
\mathbb{P}^0(S_{n}=x\pm e_1)=\sum_{\mathcal I_n\subset \{1,\ldots,n\}}  \mathbb{P}^0(S_{n}=x\pm e_1\vert A(\mathcal I_n))\, \mathbb{P}^0(A(\mathcal I_n).
$$
Therefore, it suffices to prove that if $\mathbb{P}^0(S_{n}=x+e_1\vert A(\mathcal I_n))>0$, then
\begin{align*}
	\mathbb{P}^0(S_{n}=x+e_1\vert A(\mathcal I_n))<\mathbb{P}^0(S_{n}=x-e_1\vert A(\mathcal I_n)),\qquad  n\geq 0 \textnormal{ even}.
\end{align*}To see this, we use the fact that conditional on $A(\mathcal I_n)$, we have two independent random walks: one in the $\pm e_1$ direction of $\tilde n=|\mathcal I_n|$ steps and one in the remaining $d-1$ directions of $n-\tilde n$ steps. For any $y\in\Z^d$, we write $y=(y_1,P(y))$ with $P(y)\in \Z^{d-1}$. Note that $P(x+e_1)=P(x-e_1)=P(x)$. Hence, if $\mathbb{P}^0(S_{n}=x+e_1\vert A(\mathcal I_n))>0$, then
$$
\begin{aligned}
\mathbb{P}^0(S_{n}=x+e_1\vert A(\mathcal I_n))
&=\mathbb{P}^0(P(S_{n})=P(x) \vert A(\mathcal I_n)) \binom{\tilde n}{\tfrac{\tilde n+x_1+1}{2}} 2^{-\tilde n}\\
&<\mathbb{P}^0(P(S_{n})=P(x) \vert A(\mathcal I_n)) \binom{\tilde n}{\tfrac{\tilde n+x_1-1}{2}} 2^{-\tilde n}\\
&=\mathbb{P}^0(S_{n}=x-e_1\vert A(\mathcal I_n)),
\end{aligned}
$$
where we used the elementary estimate that
$$
\binom{\tilde n}{\tfrac{\tilde n+x_1+1}{2}}<\binom{\tilde n}{\tfrac{\tilde n+x_1-1}{2}}
$$
because $\tilde n\ge x_1\geq 1$. This proves Lemma \ref{lm:G0}.
\e{proof}

\end{appendix}

\bibliographystyle{amsplain}

\end{document}